\documentclass{article}

\usepackage[latin2]{inputenc}
\usepackage{t1enc, anysize, enumerate, amsthm, amsmath, amssymb}
\newtheorem{theorem}{Theorem}[section]
\newtheorem{lemma}[theorem]{Lemma}

\newenvironment{definition}[1][Definition]{\begin{trivlist}
\item[\hskip \labelsep {\bfseries #1}]}{\end{trivlist}}
\newenvironment{conjecture}[1][Conjecture]{\begin{trivlist}
\item[\hskip \labelsep {\bfseries #1}]}{\end{trivlist}}
\newenvironment{remark}[1][Remark]{\begin{trivlist}
\item[\hskip \labelsep {\bfseries #1}]}{\end{trivlist}}
\newenvironment{notation}[1][Notation]{\begin{trivlist}
\item[\hskip \labelsep {\bfseries #1}]}{\end{trivlist}}

\begin{document}

\title{Incomparable copies of a poset in the Boolean lattice}
\author{Gyula O.H. Katona\footnote{Alfr\'ed R\'enyi Institute of Mathematics}~ and
D\'aniel T. Nagy\footnote{E\"{o}tv\"{o}s Lor\'and University, Budapest}}
\maketitle

\begin{abstract}
Let $B_n$ be the poset generated by the subsets of $[n]$ with the inclusion as relation and let $P$ be a finite poset. We want to embed $P$ into $B_n$ as many times as possible such that the subsets in different copies are incomparable. The maximum number of such embeddings is asymptotically determined for all finite posets $P$ as $\frac{{n \choose \lfloor n/2\rfloor}}{M(P)}$, where $M(P)$ denotes the minimal size of the convex hull of a copy of $P$. We discuss both weak and strong (induced) embeddings.
\end{abstract}

\section{Introduction}~

\begin{definition}
Let $B_n$ be the Boolean lattice, the poset generated by the subsets of $[n]$ with the inclusion as relation and $P$ be a finite poset with the relation $<_p$. (If $S$ is a set of size $n$ we may also write $B_S$.) $f:~P\rightarrow B_n$ is an {\it embedding} of $P$ into $B_n$ if it is an injective function that satisfies $f(a)\subset f(b)$ for all $a<_p b$. $f$ is called an {\it induced embedding} if it is an injective function such that $f(a)\subset f(b)$ if and only if $a<_p b$.
\end{definition}

\begin{definition}
Let $X$ and $Y$ be two sets of subsets of $[n]$. $X$ and $Y$ are {\it incomparable} if there are no sets $x\in X$ and $y\in Y$ such that $x\subseteq y$ or $y\subseteq x$. A family of sets of subsets of $[n]$ is {\it incomparable} if its elements are pairwise incomparable.
\end{definition}

We investigate the following problem. How many times can we embed a poset into $B_n$ such that the resulting copies form an incomparable family? An asymptotic answer is given in both the induced and the non-induced case. Before we can state our main result, some notations are needed.

\begin{notation}
Let $F\subseteq B_n$. The {\it convex hull} of $F$ is the set
\begin{equation}
conv(F)=\{b\in B_n ~ \big| ~ \exists a,c\in F ~~ a\subseteq b\subseteq c\}.
\end{equation}
We use the following notations for the minimal size of the convex hull. For a finite poset $P$
\begin{eqnarray}
t_1(P) &=& \min_{f,n}\{|conv(Im(f))| ~ \big| ~ f:~P\rightarrow B_n ~ {\rm is~an~embedding} \} \\
t_2(P) &=& \min_{f,n}\{|conv(Im(f))| ~ \big| ~ f:~P\rightarrow B_n ~ {\rm is~an~induced~embedding} \}
\end{eqnarray}
\end{notation}

\begin{theorem}\label{mainthm}
Let $P$ be a finite poset. Let $M_1(P,n)$ (and $M_2(P,n)$) denote the largest $M$ such that there are embeddings (induced embeddings) $f_1, f_2, \dots f_M: P\rightarrow B_n$ such that $\{Im(f_i),~ i=1,2,\dots M \}$ is an incomparable family. Then
\begin{eqnarray}
\lim_{n\rightarrow \infty} \frac{M_1(P,n)}{{n \choose \lfloor n/2\rfloor}} &=& \frac{1}{t_1(P)} \\
\lim_{n\rightarrow \infty} \frac{M_2(P,n)}{{n \choose \lfloor n/2\rfloor}} &=& \frac{1}{t_2(P)}.
\end{eqnarray}
\end{theorem}

We prove upper and lower bounds for $M_j(P,n)$ in the next two sections (Theorem \ref{upperthm} and Theorem \ref{lowerthm}). The two bounds will imply the theorem immediately. Since the proofs are almost identical for $j=1,2$, they will be done simultaneously.

\begin{remark}
Theorem \ref{mainthm}. was independently proved by A. P. Dove and J. R. Griggs, \cite{dg}.
\end{remark}

The problem discussed in this paper is related to the problem of determining the largest families in $B_n$ avoiding certain configurations of inclusion.

\begin{definition}
Let $P_1, P_2, \dots P_k$ be finite posets. La($n, \{P_1, \dots P_k\}$) denotes the size of the largest subset $\mathcal{F}\subset B_n$ such that none of the posets $P_i$ can be embedded into $\mathcal{F}$.
\end{definition}

Let $V_k$ denote the $(k+1)$-element poset that has a minimal element contained in the other $k$ unrelated elements. $\Lambda_k$ is obtained from $V_k$ by reversing the relations. Katona and Tarj\'an proved that a subset of $B_n$ containing none of the posets $\{V_2,~ \Lambda_2\}$ has at most ${n-1\choose \left\lfloor {n-1\over 2} \right\rfloor }$ elements, and this bound is sharp \cite{tarjan}. Such a family consists of pairwise incomparable copies of the one-element poset and the two-element chain.

Another example of the relation of the two problems is determining La($n, V_2$). (See \cite{rfork} for asymptotic bounds on La($n, V_r$).) A $V_2$-free family consists of pairwise independent copies of the posets $\{\Lambda_0, \Lambda_1, \Lambda_2, \dots \}$.

The value of La($n,P$) is not known for a general poset $P$, but many special cases have been solved. See \cite{bukh} for posets whose Hasse diagram is a tree. See \cite{diamond} for diamond and harp posets. \cite{chen} provides upper bounds on La($n,P$) for all posets $P$.

\section{The upper bound}~

To prove the upper bound for $M_j(P,n)$ we need a lemma about chains. Let $S$ be a set of size $n$. A chain in $S$ is a set of subsets $\emptyset=C_0\subset C_1\subset C_2\subset \dots \subset C_n=S$, where $|C_m|=m$ for all $m$.

\begin{lemma}\label{chcount}
Let $\mathcal{F}$ be a family of subsets of $S$, where $|S|=n$ and $|\mathcal{F}|=t$. Then the number of chains intersecting at least one member of $\mathcal{F}$ is at least
\begin{equation} \left(t-\frac{{t \choose 2}}{n}\right)\lfloor n/2\rfloor!\lceil n/2 \rceil!. \end{equation}
\end{lemma}

\begin{proof}
We prove the lemma by induction on $t$. The statement is true for $t=1$, as the number of chains passing through a subset $F$ is $|F|!(n-|F|)!\geq\lfloor n/2\rfloor!\lceil n/2 \rceil!$. Now let $t\geq 2$, and $\mathcal{F}=\{F_1, F_2, \dots F_t\}$. Since taking complements does not change the number of intersecting chains, we may assume that some set of $\mathcal{F}$ has size at most $\lfloor n/2\rfloor$. We can also assume that $F_t$ is one of the smallest subsets.

By induction, the number of chains intersecting $\mathcal{F}\backslash \{F_t\}$ is at least
\begin{equation} \left(t-1-\frac{{t-1 \choose 2}}{n}\right)\lfloor n/2\rfloor!\lceil n/2 \rceil!. \end{equation}
The number of chains through $F_t$ is $|F_t|!(n-|F_t|)!$. Assume that $F_t\subset F_i$ for some $i\in [1,~n-1]$. The number of chains intersecting both $F_t$ and $F_i$ is $|F_t|!(|F_i|-|F_t|)!(n-|F_i|)!\leq |F_t|!(n-|F_t|-1)!$. So there are at least
\begin{equation} |F_t|!(n-|F_t|)!\left(1-\frac{t-1}{n-|F_t|}\right)\geq
\lfloor n/2\rfloor!\lceil n/2 \rceil!\left(1-\frac{2(t-1)}{n}\right)\end{equation}
chains that intersect $\mathcal{F}$ only in $F_t$. The statement of the lemma follows after summation: \begin{equation} \left(t-1-\frac{{t-1 \choose 2}}{n}\right)+\left(1-\frac{2(t-1)}{n}\right)=t-\frac{{t \choose 2}}{n}. \end{equation}
\end{proof}

\begin{theorem}\label{upperthm}
For any finite poset $P$
\begin{equation}
M_j(P,n)\leq \frac{1}{t_j(P)}{n \choose \lfloor n/2\rfloor}(1+O(n^{-1}))
\end{equation}
holds for $j=1,2$.
\end{theorem}

\begin{proof}
Assume that $f_1, f_2, \dots f_k: P\rightarrow B_n$ are embeddings (induced if $j=2$) such that the family $\{Im(f_i),~ i=1,2,\dots k \}$ is incomparable. Than $\{conv(Im(f_i)),~ i=1,2,\dots k \}$ is also an incomparable family. To see that, assume there are sets $a,b$ such that $a\subseteq b$, $a\in conv(Im(f_i))$, $b\in conv(Im(f_j))$ and $i\not=j$. Then by the definition of the convex hull there are sets $a'\in Im(f_i)$ and $b'\in Im(f_j)$ such that $a'\subseteq a\subseteq b\subseteq b'$. But $a'\not\subseteq b'$ since $\{Im(f_i),~ i=1,2,\dots k \}$ is an incomparable family.

Since the family $\{conv(Im(f_i)),~ i=1,2,\dots k \}$ is incomparable, every chain intersects at most one of its members. By Lemma \ref{chcount}., each $conv(Im(f_i))$ intersects at least $t_j(P)\lfloor n/2\rfloor!\lceil n/2 \rceil!(1-O(n^{-1}))$ chains. Since the total number of chains is $n!$,
\begin{equation}
k \leq \frac{n!}{t_j(P)\lfloor n/2\rfloor!\lceil n/2 \rceil!(1-O(n^{-1}))}=\frac{1}{t_j(P)}{n \choose \lfloor n/2\rfloor}(1+O(n^{-1})).
\end{equation}
\end{proof}

\section{The lower bound}~

In this section our aim is to prove a lower bound on $M_j(P,n)$ by embedding many copies of $P$ to $B_n$. We need the following lemmas for the construction.

\begin{lemma}\label{labeling}
Let $P$ be a finite poset, and let $f: P\rightarrow B_m$ be an embedding. Then we can label the elements of $B_m$ with the numbers $1, 2, \dots 2^m$ such that all the sets get a higher number than any of their subsets, and the numbers assigned to the elements of $conv(Im(f))$ form an interval in $[1, 2^m]$.
\end{lemma}

\begin{proof}
We divide the elements of $B_m$ into three groups:

Let $\mathcal{F}_1=\{b\in B_m \big|~ \exists c\in Im(f)~~b\subset c,~~\nexists a\in Im(f)~~a\subseteq b \}$,

$\mathcal{F}_2=conv(Im(f))=\{b\in B_n ~ \big| ~ \exists a,c\in Im(f) ~~ a\subseteq b\subseteq c\}$

and $\mathcal{F}_3=B_m\backslash (\mathcal{F}_1 \cup \mathcal{F}_2)=\{b\in B_n ~ \big| ~ \nexists c\in Im(f) ~~ b\subseteq c\}$.

We use the numbers of $[1,~ |\mathcal{F}_1|]$ for the sets of $\mathcal{F}_1$, the numbers of $[|\mathcal{F}_1|+1,~ |\mathcal{F}_1|+|\mathcal{F}_2|]$ for the sets of $\mathcal{F}_2$ and the numbers $[|\mathcal{F}_1|+|\mathcal{F}_2|+1,~ 2^m]$ for the sets of $\mathcal{F}_3$. In the groups we assign numbers such that the elements representing larger subsets get larger numbers.

We have to check that if that if $x,y\in B_m$ and $y$ got a larger number than $x$, then $y\not\subset x$.

If $x$ and $y$ are in the same group, than $|x|\leq |y|$, so $y\not\subset x$.
If $x\in\mathcal{F}_1$ and $y\in\mathcal{F}_2$, then $y\not\subset x$, because $y$ contains an element of $Im(f)$ while $x$ does not.
If $x\in\mathcal{F}_1\cup \mathcal{F}_2$ and $y\in\mathcal{F}_3$, then $y\not\subset x$, because $x$ is the subset of an element of $Im(f)$ while $y$ is not.

\end{proof}

\begin{lemma}\label{ordcopies}
Let $P$ be a finite poset and let $\varepsilon' >0$ be fixed. Let $j\in\{1,2\}$. Then there are integers $N, K$ and functions $f_1, f_2, \dots f_K: P\rightarrow B_N$ such that
\begin{enumerate}[(i)]
\item For all $i\in [1, K]$, $f_i$ is an embedding if $j=1$, and an induced embedding if $j=2$.
\item $K\geq \frac{2^N(1-\varepsilon')}{t_j(P)}$.
\item If $i_1<i_2$, $a\in Im(f_{i_1})$ and $b\in Im(f_{i_2})$, then $b\not\subseteq a$.
\end{enumerate}
\end{lemma}

\begin{proof}
Let $P$ be a fixed finite poset. There is embedding (or induced embedding, if $j=2$) $f: P\rightarrow B_m$ for some $m$  such that $|conv(Im(f))|=t_j(P)$. Fix $m$ and $f$. Choose $k\in\mathbb{N}$ such that $\left(1- \frac{t_j(P)}{2^m}\right)^k\leq \varepsilon'$, and let $N=km$. Let $S_1, S_2, \dots S_k$ be disjoint sets of size $m$ and let $S=\bigcup_{i=1}^k S_i$. Consider the elements of $B_N$ as the subsets of $S$.

Let $g_i: P\rightarrow B_{S_i}$ ($i=1, 2, \dots k$) be embeddings that map the elements of $P$ to $m$-element sets the same way as $f$ does. Assign the numbers $1, 2, \dots 2^m$ to the subsets of $S_i$ as in Lemma \ref{labeling}. The elements of $conv(Im(g_i))$ will get the numbers of the interval $I=[p,~ p+t_j(P)-1]$ for all $i$.

We call an embedding $g: P\rightarrow B_S$ ${\it good}$ if there is an index $i\in [1,~k]$ and there are $k-1$ sets $A_1\subseteq S_1,~ A_2\subseteq S_2, \dots A_{i-1}\subseteq S_{i-1},~ A_{i+1}\subseteq S_{i+1}, \dots A_k\subseteq S_k $ such that none of the numbers assigned to $A_1, A_2, \dots A_{i-1}$ is in $I$, and for any $x\in P$, $g(x)\cap S_i=g_i(x)$, and $g(x)\cap (S\backslash S_i)=\displaystyle\bigcup_{r\in [n]\backslash \{i\}} A_r$.

The number of good functions is
\begin{equation}
\sum_{i=1}^{k} (2^m-t_j(P))^{i-1}\cdot(2^m)^{k-i}=2^{N-m} \sum_{i=1}^{k} \left(1-\frac{t_j(P)}{2^m}\right)^{i-1}=
\end{equation}
\begin{equation}\nonumber
2^{N-m}\frac{1-\left(1-\frac{t_j(P)}{2^m}\right)^k}{\frac{t_j(P)}{2^m}}=
\frac{2^N}{t_j(P)}\left(1-\left(1-\frac{t_j(P)}{2^m}\right)^k\right)\geq \frac{2^N(1-\varepsilon')}{t_j(P)}.
\end{equation}

$f_1, f_2, \dots f_K$ will be the good functions. They are embeddings (induced if $j=2$), and their number is sufficiently large. So $(i)$ and $(ii)$ are satisfied. Now we find an ordering of the good functions that satisfies $(iii)$.

Let $g$ be the good function defined by the index $i$ and the subsets $A_1, A_2, \dots A_{i-1}, A_{i+1}, \dots A_k$. Define the code of $g$ as a vector of length $k$ with coordinates as follows. The first $i-1$ coordinates are numbers assigned to the sets $A_1, A_2, \dots A_{i-1}$ respectively. The $i$th coordinate is $p$, the smallest number in $I$. The last $k-i$ coordinates are numbers assigned to the sets $A_{i+1}, A_{i+2}, \dots A_k$ respectively. Now take the lexicographic ordering of these codes, and assign the names $f_1, f_2, \dots$ to the good functions according to the ordering. ($f_1$ will be the good function whose code comes first in the lexicographic ordering, $f_2$ will be the second and so on.)

Now we can verify $(iii)$. Assume that $A\in Im(f_a)$, $B\in Im(f_b)$, $A\subseteq B$ and $b<a$. Let the $l$th be the first coordinate where the codes of $f_a$ and $f_b$ are different. Since $b<a$, the $l$th coordinate of $a$ is strictly larger than that of $b$, and the first $l-1$ coordinates are not from $I$. That implies that the number assigned to $A\cap S_l$ is strictly larger than the number assigned to $B\cap S_l$. (We use the fact that the numbers assigned to the elements of $conv(Im(g_l))$ form an interval at this step.) Then $A\cap S_l\not\subseteq B\cap S_l$ (contradicting $A\subseteq B$) as the labeling of the elements of $S_l$ is done according to Lemma \ref{labeling}.
\end{proof}

\begin{theorem}\label{lowerthm}
Let $P$ be a finite poset, $\varepsilon >0$ and $j\in \{1,2\}$. Then for all large enough $n$
\begin{equation}
M_j(P,n)\geq  \frac{1}{t_j(P)}{n \choose \lfloor n/2\rfloor}(1-\varepsilon).
\end{equation}
\end{theorem}

\begin{proof}
Choose $N,~K$, and $f_1, f_2, \dots f_K: P\rightarrow B_N$ as in Lemma \ref{ordcopies}. (Use $\varepsilon'=\frac{\varepsilon}{2}$). Consider the elements of $B_N$ as the subsets of a set $S$ of size $N$. Let $R$ be a set such that $S\subset R$ and $|R|=n$. Let $Q=R\backslash S$. Let
\begin{equation}
\mathcal{Q}=\left\{T\subset Q~ \Big|~ \left\lfloor\frac{n-N}{2}\right\rfloor-K \leq |T| \leq \left\lfloor\frac{n-N}{2}\right\rfloor-1 \right\}.
\end{equation}

If $n$ is large enough, then the following inequality is true:
\begin{equation}
  \sum_{i=1}^K {n-N \choose \left\lfloor\frac{n-N}{2}\right\rfloor-i} \geq K\cdot {n \choose \left\lfloor\frac{n-N}{2}\right\rfloor}\left(1-\frac{\varepsilon}{2}\right).
\end{equation}

Then
\begin{equation}
|\mathcal{Q}|\geq
K\cdot {n-N \choose \left\lfloor\frac{n-N}{2}\right\rfloor}\left(1-\frac{\varepsilon}{2}\right) \geq
\frac{2^N(1-\frac{\varepsilon}{2})}{t_j(P)} \cdot
2^{-N}{n \choose \left\lfloor\frac{n}{2}\right\rfloor}\left(1-\frac{\varepsilon}{2}\right) \geq
\frac{1}{t_j(P)}{n \choose \lfloor n/2\rfloor}(1-\varepsilon).
\end{equation}

We used that $2^N {n-N \choose \left\lfloor\frac{n-N}{2}\right\rfloor}\geq {n \choose \left\lfloor\frac{n}{2}\right\rfloor}$. It can be verified easily by induction on $N$.

We define an embedding $f_T: P\rightarrow B_R$ (induced if $j=2$) for every $T\in\mathcal{Q}$ such that $\{Im(f_T)~\big|~T\in\mathcal{Q}\}$ is an incomparable family. For any $x\in P$ let $f_T(x)\cap Q=T$ and $f_T(x)\cap S=f_{\lfloor \frac{n-N}{2} \rfloor-|T|}(x)$. Then $f_T$ is obviously an embedding (induced if $j=2$).

Now we check that the family $\{Im(f_T)~\big|~T\in\mathcal{Q}\}$ is incomparable. Let $T_1, T_2\in \mathcal{Q}$ be different sets. Assume that $A_1\in Im(f_{T_1})$, $A_2\in Im(f_{T_2})$ and $A_1\subseteq A_2$. Then $T_1=A_1\cap Q\subseteq A_2\cap Q=T_2$. Since $T_1\not= T_2$, $|T_1|<|T_2|$ holds. Since $A_1\cap S\in Im(f_{\lfloor \frac{n-N}{2} \rfloor-|T_1|})$ and $A_2\cap S\in Im(f_{\lfloor \frac{n-N}{2} \rfloor-|T_2|})$, Lemma \ref{ordcopies}. $(3)$ implies $|A_1\cap S|\not\subseteq |A_2\cap S|$. It contradicts $A_1\subseteq A_2$, so the family is indeed incomparable.

We found at least $\frac{1}{t_j(P)}{n \choose \lfloor n/2\rfloor}(1-\varepsilon)$ different embeddings  (induced if $j=2$) of $P$ to $B_R$, where $|R|=n$, such that the resulting copies form an incomparable family. It proves the theorem.
\end{proof}

\section{Remarks}~
In this section we exactly determine the maximum number of incomparable copies for certain posets. The problem has already been solved for the path posets.

\begin{theorem}\label{pathlemma} {\rm {\bf (Griggs, Stahl, Trotter) \cite{gst}}} Let $P^{h+1}$ be the path poset with $h+1$ elements. Then for all $n\geq h$
\begin{equation}M_1(P^{h+1},n)={n-h\choose \left\lfloor {n-h\over 2} \right\rfloor }.\end{equation}
\end{theorem}

We include an alternative proof for the sake of completeness. The following theorem will be used.

\begin{theorem} {\rm {\bf (Bollob\'as) \cite{B}}} Let $(A_i, B_i)~ (1\leq i\leq m)$ be a family of disjoint subsets
($A_i\cap B_i=\emptyset$), where $A_i\cap B_j\not= \emptyset$ holds for $i\not= j~ (1\leq i,j\leq m)$.
Then
\begin{equation}\sum_{i=1}^m{1\over {|A_i|+|B_i|\choose |A_i|}}\leq 1.\end{equation}
\end{theorem}

\begin{proof} (Theorem \ref{pathlemma}.)
Consider an embedding of $P^{h+1}$ into $B_n$. Let its maximal and minimal elements embedded into $C_i$ and $D_i$ respectively.  $C_i\supset D_i$ implies $\overline{C}_i\cap D_i=\emptyset.$ On the other hand, choosing these sets for all $i=1, \ldots ,m$, the incomparability conditions imply $\overline{C}_i\cap D_j\not= \emptyset$. The theorem of Bollob\'as can be applied for the pairs $(\overline{C}_i, D_i)$:
\begin{equation}\label{bollineq}
\sum_{i=1}^m {1\over {|\overline{C}_i|+|B_i|\choose |\overline{C}_i|}}\leq 1.
\end{equation}

$|C_i-D_i|\geq h$ results in $|\overline{C_i}|+|D_i|\leq n-h.$
Therefore the left hand side of (\ref{bollineq}) can be decreased in the following way.
\begin{equation}
{m \over {n-h\choose \left\lfloor {n-h\over 2} \right\rfloor }}=\sum_{i=1}^m {1 \over {n-h\choose \left\lfloor {n-h\over 2} \right\rfloor }}\leq \sum_{i=1}^m {1\over {|\overline{C}_i|+|B_i|\choose |\overline{C}_i|}}\leq 1
\end{equation}
holds, proving the upper bound in the theorem.

The lower bound can be seen by an easy construction. Let $G\subset\{h+1, h+2,\dots n\}$ be a subset of size $\left\lfloor {n-h\over 2} \right\rfloor$. Then $P^{h+1}$ can be embedded to the sets $G,~\{1\}\cup G,~\{1, 2\}\cup G,\dots ~\{1,2\dots h\}\cup G$. We have ${n-h\choose \left\lfloor {n-h\over 2} \right\rfloor }$ such embeddings and the resulting copies form an incomparable family. This proves the lower bound.
\end{proof}

\begin{definition}
Let $h(P)$ be the {\it height} of the poset $P$, that is the number of elements in a longest chain in $P$ minus 1. We say that $P$ is {\it thin} if it can be embedded into $B_{h(P)}$. $P$ is called {\it slim} if it has an induced embedding into $B_{h(P)}$.
\end{definition}

\begin{theorem}\label{thinthm}
If $P$ is a thin poset, then
\begin{equation} M_1(P,n)={n-h\choose \left\lfloor {n-h\over 2} \right\rfloor }. \end{equation}
If $P$ is slim, then
\begin{equation} M_1(P,n)=M_2(P,n)={n-h\choose \left\lfloor {n-h\over 2} \right\rfloor }. \end{equation}
\end{theorem}

\begin{proof}
Since $P^{h+1}$ is a subposet of $P$,
\begin{equation} M_2(P,n)\leq M_1(P,n) \leq M_1(P^{h+1},n). \end{equation}

Now consider $M_1(P^{h+1},n)$ incomparable copies of $P^{h+1}$ in $B_n$ as defined in Theorem \ref{pathlemma}. Their convex hulls are isomorphic to $B_h$, so we can embed $P$ to them (in an induced way if $P$ is slim). It proves $M_1(P,n) \geq M_1(P^{h+1},n)$ for thin posets, and $M_2(P,n) \geq M_1(P^{h+1},n)$ for slim posets.

We already determined the value of $M_1(P^{h+1},n)$ in Lemma \ref{pathlemma}, so the proof is completed.
\end{proof}

Of course Theorem \ref{thinthm} does not contradict Theorem \ref{mainthm}, since $t_1(P)=2^h$ and
\begin{equation}
{1\over 2^h}{n\choose \left\lfloor {n\over 2} \right\rfloor }\sim {n-h\choose \left\lfloor {n-h\over 2} \right\rfloor }.
\end{equation}

The smallest non-thin poset is $V$ with three elements, $a,b, c$ and the relations $a<b,~ a<c$.
Now we give a large set of incomparable copies for all $n$. Fix the parameter $i$ $(1\leq i\leq {n+2\over 4})$. Choose an element
\begin{equation}
F\in  {[n-2i]\choose \left\lceil {n\over 2} \right\rceil -2i+1}.
\end{equation}
Then the sets
\begin{equation}\nonumber
F\cup \{ n-(2i-3), \ldots , n\} , F\cup \{ n-(2i-3), \ldots , n\} \cup \{ n-(2i-1)\} ,
F\cup \{ n-(2i-3), \ldots , n\} \cup \{ n-(2i-2)\}
\end{equation}
form an embedding of the poset $V$. Let ${\cal P}_i$ denote the set of all such copies.
It is trivial that the copies in ${\cal P}_i$ are incomparable. But not much more difficult to check
that two copies chosen from  ${\cal P}_i$ and ${\cal P}_j\ (1\leq i<j \leq {n+2\over 4}) $, respectively,
are also incomparable. Therefore
\begin{equation}
\bigcup_{i=1}^{ \left\lfloor {n+2\over 4}\right\rfloor }{\cal P}_i
\end{equation}
is a collection of incomparable embeddings of $V$. We conjecture that this is the largest one.
\begin{conjecture}
\begin{equation}
M_1(V,n)=\sum_{i=1}^{ \left\lfloor {n+2\over 4}\right\rfloor }{n-2i\choose \left\lceil {n\over 2} \right\rceil -2i+1}.
\end{equation}
\end{conjecture}

\end{document}